\documentclass[11pt, reqno]{amsart}

\usepackage {amsmath, amssymb, amscd, mathrsfs, url,verbatim,enumerate}
\usepackage[text={6.5in,9in},centering,letterpaper,dvips]{geometry}
\usepackage{multirow,latexsym,graphicx,epstopdf,comment}
\usepackage[all]{xy}

\usepackage[dvipsnames,usenames]{color}
\usepackage[colorlinks=true, urlcolor=NavyBlue, linkcolor=NavyBlue, citecolor=NavyBlue]{hyperref}
\usepackage{graphicx}
\usepackage{epsfig}
\usepackage[latin1]{inputenc}
\usepackage{amsmath}
\usepackage{amsfonts}
\usepackage{amssymb}
\usepackage{amsthm}
\usepackage{amscd}
\usepackage{verbatim}
\usepackage{subfigure}
\usepackage{pinlabel}
\usepackage{stmaryrd}
\usepackage{mathtools}
\usepackage{enumerate}

\usepackage{tikz}
\usepackage{nicefrac}
\usetikzlibrary{arrows}
\usetikzlibrary{decorations.pathreplacing}
\usepackage{verbatim}
	
\newtheorem{theorem}{Theorem}[section]

\newtheorem{lemma}[theorem]{Lemma}
\newtheorem{proposition}[theorem]{Proposition}

\newtheorem{corollary}[theorem]{Corollary}

\theoremstyle{definition}

\theoremstyle{remark}
\newtheorem{remark}[theorem]{Remark}

\def\Z{\mathbb{Z}}

\def\F{\mathbb{F}}

\def\deg{\textup{deg}}
\def\vol{\textup{vol}}

\title{Heegaard Floer homology and splicing homology spheres}
\author{\c{C}a\u{g}ri Karakurt}
 \address{Department of Mathematics, Bo\u{g}azi\c{c}i University, Bebek 34342}
   \thanks{\c{C}K was supported by BAGEP award of the Science Academy and  Bo\u{g}azi\c{c}i University Research Fund Grant Number 12482}
\email{{cagri.karakurt@boun.edu.tr}}

\author{Tye Lidman}
 \address{Department of Mathematics, North Carolina State University\\Raleigh, NC 27695}
   \thanks{{TL was supported by NSF grant DMS-1709702 and a Sloan Fellowship}}
\email{{tlid@math.ncsu.edu}}

\author{Eamonn Tweedy}
 \address{Department of Mathematics, Widener University, Chester, PA 19013}
\email{{etweedy@widener.edu}}

\begin{document}
\maketitle

\begin{abstract}
We prove a basic inequality for the $d$-invariants of a splice of knots in homology spheres.  As a result, we are able to prove a new relation on the rank of reduced Floer homology under maps between Seifert fibered homology spheres, improving results of the first and second authors.  As a corollary, a degree one map between two aspherical Seifert homology spheres is homotopic to a homeomorphism if and only if the Heegaard Floer homologies are isomorphic.  
\end{abstract}

\section{Introduction}
The Heegaard Floer homology package, due to Ozsv\'ath and Szab\'o \cite{os:disk},  \cite{OzSz1}, has become an extraordinarily valuable set of tools in the study of $3$-manifolds, $4$-manifolds, knots, and links.   Although the Heegaard Floer $3$-manifold invariants behave in a straightforward way with respect to the operation of connected sum, it is much more difficult to determine how they behave under the more general operation of splicing two $3$-manifolds along knots.  Recall that for two knots $K_0, K_1$ in homology spheres $Y_0, Y_1$, the \emph{splice} is the manifold obtained by gluing the exteriors of the $K_i$ by identifying the meridian of $K_0$ with the longitude of $K_1$ and vice versa.  See \cite{Sav} for a further exposition on the splicing construction.  In the current article, we study some properties of the Heegaard Floer homology of a splice and apply this to the Heegaard Floer homology of Seifert homology spheres.

We shall focus on two particular components of the Heegaard Floer homology package, the {\em reduced Heegaard Floer homology} $HF_{red}$ and the {\em correction term} $d$.  For the weaker {\em hat Heegaard Floer homology} $\widehat{HF}$, the splice can be studied using bordered Floer homology \cite{LOT} - see, for instance, the work of Hedden-Levine \cite{HeddenLevine}, Hanselman \cite{Hanselman} and Hanselman-Rasmussen-Watson \cite[Figure 3]{HRW} - but does not have the full power of the Heegaard Floer theory.   Throughout the present article we assume that $Y$ is an integral homology sphere, but the following Heegaard Floer constructions also work for rational homology spheres.  The \emph{reduced Heegaard Floer homology} $HF_{red}(Y)$ is a finitely generated module over $\F[U]$, where $\F = \Z/2\Z$,  which is a quotient of the \emph{plus Heegaard Floer homology} module $HF^+(Y)$; the dimension of $HF_{red}$ over $\F$ can be viewed as some measure of the complexity of $Y$ \cite{os:disk}.  In particular, $HF_{red}(Y) = 0$ if and only if $Y$ is an \emph{$L$-space}, i.e. $Y$ has the same Heegaard Floer homology as a lens space.  In this case, $Y$ is obstructed from admitting a co-orientable taut foliation or admitting a symplectic filling with $b^+ > 0$ \cite{OS04}.  It should be noted that the only Seifert fibered homology spheres which are $L$-spaces are $S^3$ and the Poincar\'e homology sphere with either of its orientations, and it is not known whether there are other prime $L$-space homology spheres aside from these.

The \emph{correction term} or \emph{$d$-invariant} $d(Y)$ is an even integer which can be used to obstruct $Y$ from bounding certain types of $4$-manifolds \cite{OzSz1}.  In particular, $d$ is a surjective homomorphism from the integer homology cobordism group to $2\mathbb{Z}$ and so if $d(Y) \neq 0$ one can conclude that $Y$ is of infinite order in homology cobordism.  The $d$-invariants have been used in many applications, including Dehn surgery problems, knot concordance, and computing the unknotting numbers of knots.  It is shown in \cite{OzSz1}  that the following relationship holds among $HF_{red}$, $d$, and Casson's invariant $\lambda$: 
\begin{equation}\label{eq:casson}
\frac{d(Y)}{2} + \lambda(Y) = \chi(HF_{red}(Y)).
\end{equation}
If $Y$ is merely a rational homology sphere, the $d$-invariant associates to each $Spin^c$ structure on $Y$ a rational number and these numbers are $Spin^{c}$-rational homology cobordism invariants.

The $d$-invariant is known to be additive under connected sum \cite{OzSz1} in general.  However,  it is not known how the $d$-invariant behaves with respect to the operation of splicing homology spheres along knots, let alone gluings of knot exteriors in general.   Our first result is an estimate for the $d$-invariants of certain gluings of exteriors of knots in homology spheres in terms of the $d$-invariants of surgeries on those knots.  To state the result, we fix our terminology.   If $K$ is a knot in an integral  homology sphere $Y$, let us denote the result of $\nicefrac{1}{n}$-framed surgery on $Y$ along $K$ by $Y(K;\nicefrac{1}{n})$.  Suppose $K_1$ and $K_2$ are knots in homology spheres $Y_1$ and $Y_2$, respectively. For  integers $n_1$ and $n_2$, let $A$ be the matrix 
\begin{align}\label{eqn:matrixA}
	A = \begin{pmatrix} -n_1 & 1 \\ 1 - n_1 n_2 & n_2 \end{pmatrix}.
\end{align}
  Let $Y(K_1,K_2, \pm A)$,  denote the manifolds resulting from gluing the exterior of $K_1$ to the exterior of $K_2$ by a map which identifies the first homology of the boundary tori by the matrix $\pm A$,  where the bases are given by the usual meridian-longitude coordinates. Although $Y_1$ and $Y_2$ are suppressed in the notation $Y(K_1,K_2,A)$, they should be inferred from where $K_1$ and $K_2$ are assumed to live in each instance. While it may seem like the  gluing matrices $\pm A$  are quite specific, they constitute precisely the set of  gluings of knot exteriors which result in homology spheres - this follows from  \cite[Lemma 1]{Gordon}. When $n_1=n_2=0$, the homology sphere $Y(K_1,K_2,A)$ is the splice of $Y_1$ and $Y_2$ along $K_1$ and  $K_2$.

\begin{theorem}\label{thm:knot-exteriors} We have
$$ \sum_{i=1}^2 d(Y_i(K_i,\nicefrac{1}{(n_i+1)})) \leq d(Y(K_1,K_2, \pm A)) \leq \sum_{i=1}^2 d(Y_i(K_i,\nicefrac{1}{(n_i-1)})).$$
\end{theorem}

Recall that Rasmussen \cite{RasmussenThesis} and Ni-Wu \cite{niwu} defined a family of invariants $V_k$ of a knot $K \subset S^3$ which are in fact invariants of the doubly filtered chain homotopy type of the knot Floer complex of $K$.  In the case of gluing the exteriors of two knots in $S^3$, we can estimate and often directly calculate the resulting $d$-invariants using the invariant $V_0$ of the knots and their mirrors.

\begin{proposition}\label{prop:knot-exteriors2}
Let $K_1$ and $K_2$ be knots in $S^3$, let $n_1,n_2$ be integers, and let $A$ be the matrix in \eqref{eqn:matrixA}. Then the following hold:
\begin{enumerate}[(i)]
\item\label{it:item1} If $\epsilon_i^{\pm}$ denotes the sign of $n_i \pm 1$, then
\begin{equation}\label{eq:item1}
 -2 \sum_{i=1}^2 \epsilon^{+}_{i} V_0( \epsilon^{+}_{i} K_i) \leq d(Y(K_1,K_2, \pm A)) \leq  -2 \sum_{i=1}^2 \epsilon^{-}_{i} V_0( \epsilon^{-}_{i} K_i).
\end{equation}
\item\label{it:item2} In particular, if $|n_1|, |n_2| \geq 2$, then 
\begin{equation}\label{eq:item2}
d(Y(K_1,K_2,\pm A)) = -\epsilon_1 2V_0(\epsilon_1 K_1) - \epsilon_2 2 V_0(\epsilon_2 K_2),
\end{equation}
where $\epsilon_i$ denotes the sign of $n_i$.
\item\label{it:item3} If $K_1$ and $K_2$ have the property that $V_0(K_i) = V_0(-K_i) = 0$, then $d(Y(K_1,K_2, \pm A)) = 0$ for any values of $n_1$,$n_2$.
\end{enumerate}
\end{proposition}

The condition involving $V_0$ in part \eqref{it:item3}  of Proposition~\ref{prop:knot-exteriors2} holds in particular when $K_1$ and $K_2$ are slice knots. It is interesting to compare this with Gordon's result \cite[Corollary 3.1]{Gordon} which says that the gluing $Y(K_1,K_2,\pm A)$ in fact bounds a contractible $4$-manifold if either both  $K_1$ and $K_2$ are  slice knots or for one of $i = 1, 2$, both $K_i$ is slice and $n_i=0$. From this it immediately follows that $d(Y(K_1,K_2,\pm A)) = 0$.    Even though our proof using Lemma~\ref{lem:destimate} formally does not rely on Gordon's result, it uses his observation, namely  Lemma~\ref{lem:desc} below, which first appeared in  \cite{Gordon}. 

While the above results follow from a direct implementation of standard properties of $d$-invariants, we will also provide some novel applications to Seifert fibered homology spheres.  Recall that a collection $a_1, \ldots, a_n$ of $n \geq 3$ pairwise coprime positive integers determines the Seifert fibered integer homology sphere $Y=\Sigma(a_1, \ldots, a_n)$ as follows: There exist unique integers $b_1, \ldots, b_n$ and $e$ satisfying $ a_1 \ldots a_n \left( \sum_{j = 1}^n \frac{b_j}{a_j} - e \right) = - 1$ and $a_j>b_j\geq 0$ for all $j=1,\dots,n$.  Note that if $a_j = 1$, then $b_j = 0$.  We start with an $S^1$-bundle over $S^2$ with Euler number $-e$. Take $n$ distinct  fibers and perform $-\nicefrac{a_j}{b_j}$-framed Dehn surgery for  $j=1,\dots n$ to obtain $Y$.  This manifold also admits a natural $S^1$-action the orbits of which are still 
called fibers. If $a_j > 1$, then the fiber through the longitude of $-\nicefrac{a_j}{b_j}$-framed surgery has nontrivial isotropy of order $a_j$, and this fiber is called the singular Seifert fiber of order $a_j$.   All other fibers are called regular fibers.  If $Y$ has fewer than three singular fibers, then it is diffeomorphic to $S^3$.  Up to possible reversing orientation, any Seifert fibered space which is a homology sphere can be obtained in the above way.  Note that our conventions are chosen so that  $Y$ bounds a negative  definite plumbed $4$-manifold $W$ and we call this the positive Seifert orientation.  See \cite{LidTwe} for more details on this convention.

It was shown in \cite{os:plumb} how $d(Y)$ can be obtained from an elementary calculation involving the intersection form on $W$.  In addition to the applications we now describe, in Section~\ref{s:exam}, this technique is used in conjunction with the inequalities in Theorem~\ref{thm:knot-exteriors} above to calculate the precise value of the $d$-invariant of some graph manifold homology spheres.

Suppose now that $Y$ is a splice of any two positively oriented Seifert fibered homology spheres along knots which are Seifert fibers.  In this case, combining the right hand inequality of Theorem~\ref{thm:knot-exteriors} with Proposition~\ref{pro:LT} below immediately results in the following: 
\begin{corollary}\label{thm:splice}
Let $Y_1$ and $Y_2$ be two positively oriented Seifert fibered homology spheres and let $Y$ denote the splice of $Y_1$ and $Y_2$ along Seifert fibers of $Y_1$ and $Y_2$.  Then, we have 
\begin{equation}\label{eq:splice-d}
d(Y_1) + d(Y_2) \geq d(Y).
\end{equation}
\end{corollary}

This has an immediate corollary using \eqref{eq:casson}.  Since $\chi(HF_{red}(Y_i)) = - \dim HF_{red}(Y_i)$ by our orientation conventions, and the Casson invariant is splice-additive \cite{FM}, we deduce 
\[
\chi(HF_{red}(Y)) \leq - \dim HF_{red}(Y_1)  - \dim HF_{red}(Y_2).
\]
Consequently, we have:
\begin{corollary}\label{cor:splice-red}
Let $Y_1$ and $Y_2$ be two positively oriented Seifert fibered homology spheres and let $Y$ denote the splice of $Y_1$ and $Y_2$ along Seifert fibers of $Y_1$ and $Y_2$.  Then, 
\begin{equation}\label{eq:splice-red}
\dim HF_{red}(Y) \geq \dim HF_{red}(Y_1) + \dim HF_{red}(Y_2).  
\end{equation}
\end{corollary}
This should be compared to the work of Hanselman-Rasmussen-Watson \cite{HRW}, who prove that $\dim \widehat{HF}(Y) \geq \dim \widehat{HF}(Y_i)$ for an arbitrary splice.  Note that we do not require that $Y$ be a Seifert fibered homology sphere.  On the other hand, when $Y$ {\em is} a Seifert fibered homology sphere, we can make a stronger connection between the Floer homology and the topology.  In this case, we have that $Y = \Sigma(a_1,\ldots, a_n)$ and $Y_1 = \Sigma(a_1,\ldots,a_k, a_{k+1} \cdots a_n)$ and $Y_2 = \Sigma(a_1 \cdots a_k, a_{k+1}, \ldots a_n)$.  Note that there is a degree one map from $Y$ to $Y_1$ (respectively $Y_2$) by pinching the piece which is Seifert over $D^2(a_{k+1},\ldots,a_n)$ (respectively $D^2(a_1,\ldots,a_k)$) to a solid torus.  Conversely, every degree one map between aspherical Seifert fibered homology spheres is homotopic to a composition of such maps by \cite{Rong}.  Along these lines, the first and second authors \cite{KarLid} proved that if $f:Y \to Y_0$ is a non-zero degree map between Seifert homology spheres, then 
\begin{equation}\label{eq:degree}
\dim HF_{red}(Y) \geq |\deg(f)| \dim HF_{red}(Y_0).  
\end{equation}
Equation~\eqref{eq:splice-red} yields a substantial improvement to this inequality in the case that $f$ is degree one.  This should be reminiscent of hyperbolic volume: if $f:Y \to Y_0$ is a non-zero degree map between hyperbolic $3$-manifolds, then $\vol(Y) \geq |\deg(f)| \vol(Y_0)$.  Further, in the case of hyperbolic volume, $f$ is homotopic to a covering if and only if equality holds.  We can prove an analogue of this for Seifert homology spheres.  (For an analogue of volume for Seifert manifolds see \cite{Rong2}.) 

\begin{theorem}\label{thm:nonzerodeg}
Let $f: Y \to Y_0$ be a non-zero degree map between aspherical Seifert fibered homology spheres.  If $\dim HF_{red}(Y) = |\deg(f)| \dim HF_{red}(Y_0)$, then $f$ is homotopic to a fiber-preserving branched covering.  In particular, if $\dim HF_{red}(Y) = \dim HF_{red}(Y_0)$, then $f$ is homotopic to a homeomorphism.        
\end{theorem}

Note that there exists a degree one map from $\Sigma(2,3,5)$ to $S^3$, both of which have trivial $HF_{red}$.

\begin{remark}
If $f: Y \to Y_0$ is a fiber-preserving branched cover, sometimes we have the equality $\dim HF_{red}(Y) = |\deg(f)| \dim HF_{red}(Y_0)$, such as for $\Sigma(2,3,35)$, which is a 5-sheeted cover branched over the singular fiber of order 7 in $\Sigma(2,3,7)$.  (The dimensions of $HF_{red}$ are 5 and 1 respectively.)  However, this is not always the case, as for $\Sigma(3,5,28)$, which is a 4-sheeted cover branched over the singular fiber of order 7 in $\Sigma(3,5,7)$.  (The dimensions are 14 and 3 respectively.)
\end{remark}

Recall that Hendricks, Hom, and the second author \cite{HHL} make use of the constructions in \cite{HM} to define the \emph{connected Heegaard Floer homology} $HF_{conn}$, which is isomorphic to a summand of $HF_{red}$ and  is a homology cobordism invariant. 

\begin{theorem}\label{thm:hfconn}
Suppose $Y$, $Y_1$, and $Y_2$ are as described in the statement of Theorem \ref{thm:splice}.
Suppose further that $HF_{conn}(Y_1) = HF_{conn}(Y_2) = 0$ and  that $Y$ is Seifert.  Then, $HF_{conn}(Y) = 0$.    
\end{theorem}

By repeating the arguments below, one can also formulate a variant of Proposition~\ref{prop:knot-exteriors2}\eqref{it:item2} for the gluing of the exteriors of fibers in Seifert homology spheres, but we leave this to the reader.  

\subsection*{Outline} In Section~\ref{s:proofs} we prove Theorem~\ref{thm:knot-exteriors} by way of two lemmas, and then prove Proposition~\ref{prop:knot-exteriors2}.  In Section~\ref{s:seifert} we prove the results pertaining to Seifert homology spheres.  Finally, in Section~\ref{s:exam}, we provide some example computations of $d$-invariants of splices.

\section{$d$-invariants, splicing, and surgery}\label{s:proofs}
We begin this section with  a proof of Theorem~\ref{thm:knot-exteriors}.  However, before proving the result in full generality, we will prove it for the special case of a splice.  First we need a key fact from \cite[p.10]{Sav}, see also \cite{FM}, \cite{Gordon}.  We include a Kirby calculus proof for the benefit of the reader.  
\begin{lemma}\label{lem:desc}
The splice $Y$ of $(M,K)$ and $(Z,J)$ can be described in the following two ways: For $\epsilon \in \{-1,+1\},$ let $M_\epsilon=M(K;\epsilon)$ and $Z_\epsilon=Z(J;\epsilon)$.  Then  $Y=(M_{-1} \# Z_{-1})(\widetilde{K}\#\widetilde{J};+1)=(M_{+1} \# Z_{+1})(\widetilde{K}\#r(\widetilde{J});-1)$  where  $\widetilde{K}$ and $\widetilde{J}$ are the canonical longitudes of the two surgeries, and $r(\widetilde{J})$ denotes the reverse of the oriented knot $\widetilde{J}$. 
\end{lemma}

\begin{proof}
From the definition, when we splice the homology spheres $M$ and $Z$ along oriented knots $K$ and $J$ we glue $M\setminus K$ to $Z\setminus J$ in such a way that the meridian of $K$ is identified with the longitude of $J$ and vice versa.  Note further that the exterior of $K$  can be seen by removing a meridian of $K$ after performing $0$-surgery on $K$, and the Seifert longitude of $K$ is then described by a meridian of this meridian, and the same is true for $J$.  
Hence the   splice is represented by the surgery diagram on top  of Figure~\ref{fig:splice_as_surgery}, where we additionally indicated the orientations on $K$ and $J$. Then the  two 0-framed unknots in this diagram can immediately be canceled to get the surgery picture in the second row. This agrees with the convention \cite[Figure 1.4]{Sav} 

\begin{figure}[h]
	\includegraphics[width=0.7\textwidth]{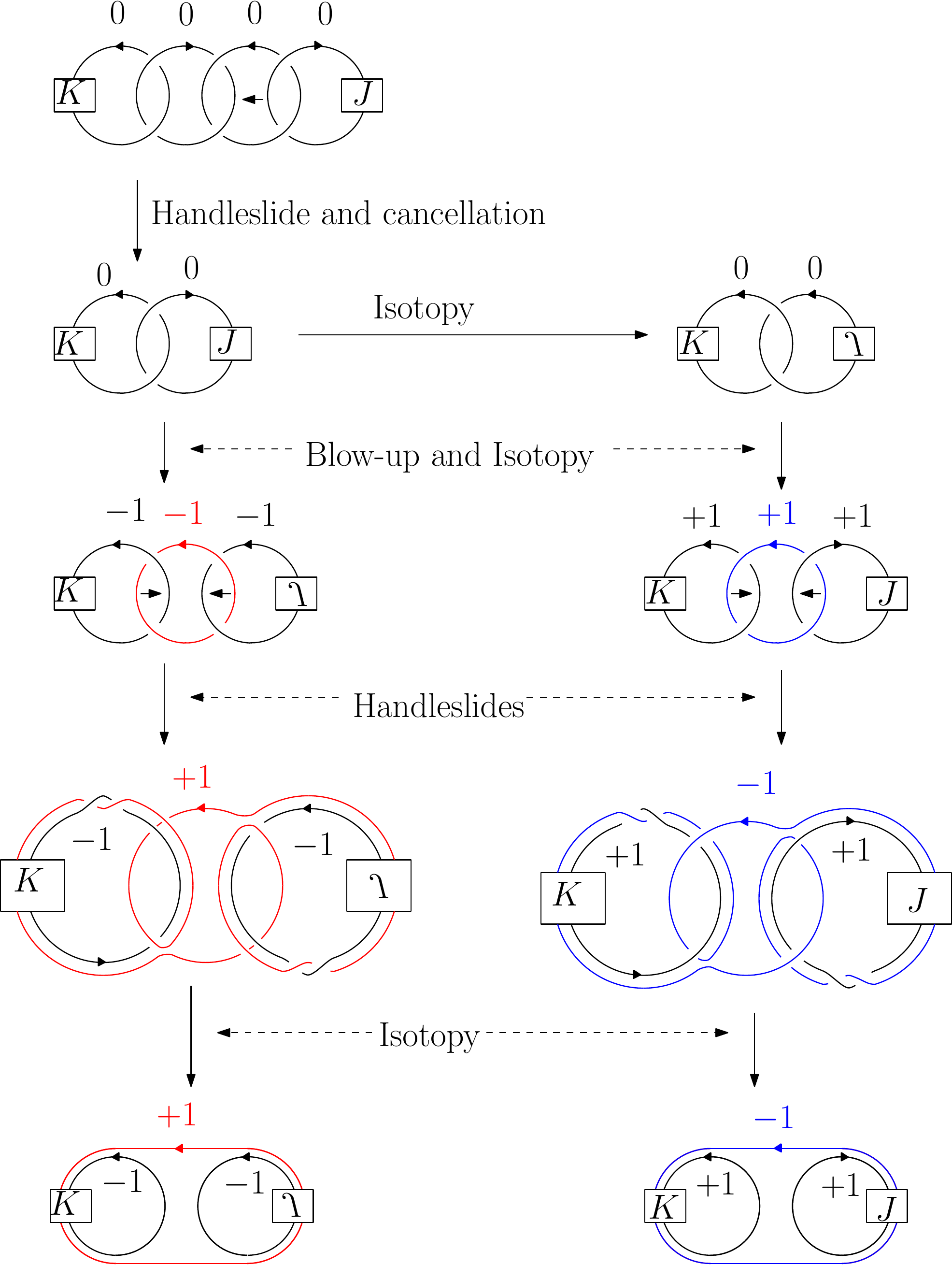}
	\caption{Surgery diagrams yielding equivalent descriptions of the splicing operation. The framings are given relative to Seifert framings.}
	\label{fig:splice_as_surgery}
\end{figure}

	To obtain the description corresponding to $\epsilon=-1$ we follow the vertical moves on the left hand side of  Figure~\ref{fig:splice_as_surgery}.  We blow-up with a $-1$-framed unknot (red curve) to unlink $K$ and $J$ then slide it over both $K$ and $J$ to get the final figure on bottom left. The red curve  now is now isotopic to $\widetilde{K}\#\widetilde{J}$  in $M_{-1}\# Z_{-1}$. 
	
The description corresponding to $\epsilon=+1$ follows from a similar argument indicated on the right hand side. Note that by an isotopy we can turn the $+1$ full twist between $K$ and $J$ into a $-1$ full twist, so we can unlink them this time by blowing-up with a $+1$-framed unknot (blue curve). We slide it over both $K$ and $J$ to  get the required description. 
\end{proof}

\begin{comment}
The second ingredient  is the following fact which should be well-known among the experts. 
\begin{lemma}\label{lem:surg} 
If $Z_1$ is an integral homology sphere and if $Z_2$ is obtained from $Z_1$ by $\epsilon\in\{-1,+1\}$-surgery on a knot $L$ in $Z_1$ then $\epsilon d(Z_2)\leq  \epsilon d(Z_1)$. 
\end{lemma}
\begin{proof}[Proof of the Lemma~\ref{lem:surg}]
Suppose $\epsilon=+1$. The surgery on $L$ corresponds to a $4$-dimensional positive-definite cobordism $W$ whose oriented boundary is the disconnected $3$-manifold $\partial W=Z_2\sqcup-Z_1$.  Changing the orientation of $W$ and attaching a $1$-handle connecting its boundary components, we see that $-Z_2\# Z_1$ bounds a negative-definite $4$-manifold. By \cite[Corollary 9.8]{OzSz1}, we have $0\leq d(-Z_2\# Z_1)=-d(Z_2)+d(Z_1)$. When $\epsilon =-1$ the cobordism $W$ is negative definite, so $Z_2\#-Z_1$ bounds a negative-definite $4$-manifold. Again by  \cite[Corollary 9.8]{OzSz1}, we have $0\leq d(Z_2\#- Z_1)=d(Z_2)+-d(Z_1)$
\end{proof}
\end{comment}

\begin{lemma}\label{lem:destimate}
Suppose $K_1$ and $K_2$ are knots in integral homology spheres $Y_1$ and $Y_2$ respectively. Let $Y$ be the splice of $Y_1$ and $Y_2$ along $K_1$ and $K_2$. Then
$$ \sum_{i=1}^2 d(Y_i(K_i;+1)) \leq d(Y) \leq \sum_{i=1}^2 d(Y_i(K_i;-1)). $$	 
\end{lemma}

\begin{proof}
First, fix $\epsilon \in \{-1, +1\}$.  Recall from \cite[Corollary 9.14]{OzSz1} that if $Z_1$ is an integral homology sphere and if $Z_2$ is obtained from $Z_1$ by $\epsilon$-surgery on a knot $L$ in $Z_1$, then $\epsilon d(Z_2)\leq  \epsilon d(Z_1)$.  Now, suppose that $Y$ is the splice of $(Y_1,K_1)$ and $(Y_2,K_2)$.  Immediately from Lemma~\ref{lem:desc} and the additivity of the $d$-invariant under connected sum, we get that for $\epsilon \in \{-1,1\}$, 
\[\epsilon d(Y) \leq \epsilon d(Y_1(K_1,\epsilon)\# Y_2(K_2,\epsilon) ) =  \epsilon (d(Y_1(K_1,\epsilon)) + d(Y_2(K_2,\epsilon) )).\] 
\end{proof}

With these lemmas in hand, we can easily prove Theorem~\ref{thm:knot-exteriors}.
\begin{proof}[Proof of Theorem~\ref{thm:knot-exteriors}]
The gluing map with matrix $A$ identifies $\mu_1 + n_1 \lambda_1$ with $\lambda_2$ and $\mu_2 + n_2\lambda_2$ with $\lambda_1$.  Note that $\mu_i + n_i \lambda_i$ (respectively $\lambda_i$) is the meridian (respectively longitude) of the core $\widetilde{K}_i$ in the surgered manifold $Y_i(K_i, \nicefrac{1}{n_i})$.  Therefore, we see that $Y(K_1,K_2,A)$ is exactly the splice of $(Y_1(K_1, \nicefrac{1}{n_1}), \widetilde{K}_1)$ and $(Y_2(K_2, \nicefrac{1}{n_2}), \widetilde{K}_2)$.  We would like to apply Lemma~\ref{lem:destimate}.  Note that the further surgered manifold $Y_i(K_i, \nicefrac{1}{n_i})(\widetilde{K}_i,\pm 1)$ is simply $Y_i(K_i, \nicefrac{1}{(n_i\pm 1)})$.  The result now follows from Lemma~\ref{lem:destimate}.

On the other hand, the gluing map with matrix $-A$ identifies $-\mu_1-n_1\lambda_1$ with $\lambda_2$ and $\mu_2 + n_2 \lambda_2$  with $-\lambda_1$. Since $-\mu_1$ and $-\lambda_1$ are the preferred meridian and longitude of the reverse knot $r(K_1)$, we can see that in this case $Y(K_1,K_2,-A)$ is the splice of $(Y_1(r(K_1), \nicefrac{1}{n_1}), r(\widetilde{K}_1))$ and $(Y_2(K_2, \nicefrac{1}{n_2}), \widetilde{K}_2)$.  Since the $d$-invariant of a surgery does not depend on string orientation of the knot, the rest of the argument is identical to that in the first case
\end{proof}

\begin{proof}[Proof of Proposition~\ref{prop:knot-exteriors2}]
\eqref{it:item1} By work of Ni and Wu \cite[Proposition 1.6]{niwu}, for any knot $K$ in $S^3$ and integer $n$ it is the case that $d(S^3(K, \nicefrac{1}{n})) = -2\epsilon V_0(\epsilon K)$ where $\epsilon$ denotes the sign of $n$.
%In particular, the $d$-invariant of $1/n$-surgery depends on $n$ only through its sign.
Inequality \eqref{eq:item1} then follows from Theorem~\ref{thm:knot-exteriors} by taking the special case $Y_i=S^3$ for $i=1,2$.

\eqref{it:item2} If additionally $|n_i| \geq 2$, then notice that $n_i \pm 1$ both have the same sign as $n_i$.  Therefore in this case, the upper and lower bounds for $d(Y(K_1,K_2,\pm A))$ provided by \eqref{eq:item1} are both equal to  $-\epsilon_1 2V_0(\epsilon_1 K_1) - \epsilon_2 2 V_0(\epsilon_2 K_2)$ and \eqref{eq:item2} follows.

\eqref{it:item3} A consequence of \cite[Theorem 2.5]{niwu} is that if $V_0(K_i) = V_0(-K_i) = 0$,  then $d(S^3(K_i,+1)) = d(S^3(K_i,-1)) = 0$.  The result then follows from \eqref{eq:item1}.
\end{proof}

%\begin{proof}[Proof of Corollary~\ref{cor:slice}]
%	Since $V_0(K_i) = V_0(-K_i) = 0$, \cite[Theorem 2.5]{niwu} implies that $d(S^3(K_i,\epsilon)=0$, for all $\epsilon \in \{-1,1\}$ and for all $i\in\{1,2\}$.  We are done by Lemma~\ref{lem:destimate}.
%\end{proof}

\section{Splicing Seifert fibers}\label{s:seifert}
We now focus on the splicing of Seifert fibers in Seifert homology spheres to prove the remaining results in the introduction.  Let $Y$ be a Seifert fibered integral homology sphere and let $W$ be the negative-definite plumbed $4$-manifold bounded by $Y$.  The intersection form on $W$ provides a unimodular negative-definite integral lattice $L_W$, and in this case $d(Y)$ can be interpreted as an invariant of the lattice by work of Ozsv\'ath-Szab\'o \cite{os:plumb}.  We briefly review the formula.

Given any unimodular negative-definite integral lattice $L$, recall that the characteristric coset of the lattice is the set of vectors $\chi \in L$  such that $\langle \chi, y \rangle_{L}\equiv \langle y,y \rangle_{L }\pmod 2$ for all $y \in L$.  Then the \emph{lattice $d$-invariant} is
$$ d(L) = \text{max} \left\{ \frac{\langle \chi , \chi \rangle_L + \text{rank}(L)}{4} : \chi \in \text{Char}(L) \right\} \in 2\mathbb{Z}.$$
It is straightforward to see that $d(L)$ is additive under lattice direct sum and that any diagonalizable lattice has $d = 0$.  Ozsv\'ath and Szab\'o prove that $d(Y) = d(L_W)$.  

In light of Lemma~\ref{lem:destimate}, to analyze the $d$-invariants of a splice of Seifert fibers, we must understand surgery on these knots.  The following is contained in work of the second and third authors \cite[Proposition 4.4]{LidTwe}, but we present a streamlined proof here.  

\begin{proposition}\label{pro:LT}
Let $Y'$ be obtained by $-1$-surgery on a Seifert fiber in a positively oriented Seifert homology sphere $Y$.  Then $d(Y') = d(Y)$.  
\end{proposition}
\begin{proof}
Consider the negative-definite star-shaped plumbed $4$-manifold $W$ bounded by $Y$.  Attach a $-1$-framed 2-handle to $W$ along the Seifert fiber.  The resulting $4$-manifold $W'$ is a negative-definite plumbing bounded by $Y'$.  Since the two-handle attachment is along the homology sphere $Y$, there is a splitting of lattices $L_{W'} \cong L_{W} \oplus (-\Z)$.  It follows that 
$$ d(Y') = d(L_{W'}) = d(L_W) + d(-\Z) =  d(Y).$$
\end{proof}

%\begin{proof}[Proof of Corollary~\ref{thm:splice}]
%	The result follows from the right hand side inequality in Lemma~\ref{lem:destimate} and  Proposition~\ref{pro:LT}
%\end{proof}

\begin{proof}[Proof of Theorem~\ref{thm:nonzerodeg}]
By work of Rong \cite{Rong}, a non-zero degree map between aspherical Seifert homology spheres is homotopic to a composition of vertical pinches, denoted $g$, followed by a fiber-preserving branched covering.  By \cite{Rong2} and \cite[Proposition 8.3]{KarLid}, $g$ is obtained by a sequence of vertical pinches of the form $\Sigma(a_1,\ldots,a_n)$ to $\Sigma(a_1 \cdots a_k,a_{k+1},\ldots,a_n)$, as described in the introduction.  We would like to show that equality in \eqref{eq:degree} implies that these vertical pinch maps are trivial (i.e. $k = 1$), which implies that $g$ is a fiber-preserving homeomorphism, completing the proof.  

If the total rank inequality in \eqref{eq:degree} is an equality, that means that it is also equality for each of the vertical pinches.   Recall that $Z = \Sigma(a_1,\ldots,a_n)$ is a splice of $Z_1 = \Sigma(a_1 \cdots a_k,a_{k+1},\ldots,a_n)$ and $Z_2 = \Sigma(a_1,\ldots, a_k, a_{k+1} \cdots a_n)$ with the same orientation, spliced along the singular fibers of order $a_1 \cdots a_k$ and $a_{k+1} \cdots a_n$ respectively.  If $\dim HF_{red}(Z) = \dim HF_{red}(Z_1)$, then $Z_2$ must be an $L$-space by \eqref{eq:splice-red}.  Note that in order to have a non-trivial pinch map, we must have $n \geq 4$ and thus $\Sigma(a_1,\ldots,a_k, a_{k+1} \cdots a_n)$ is not $S^3$ or $\Sigma(2,3,5)$, which are the only Seifert $L$-space homology spheres \cite{Eftekhary}.  Consequently, we see that the pinch maps are trivial, so $g$ is homotopic to a homeomorphism.  
\end{proof}

\begin{proof}[Proof of Theorem~\ref{thm:hfconn}]
First, we recall from \cite{DM} that a positively oriented Seifert manifold has $\underline{d} = -2\overline{\mu}$ and $d = \overline{d}$, where $\overline{\mu}$ denotes the Neumann-Siebenmann invariant \cite{Neu2},\cite{Sieb} and $\overline{d}, \underline{d}$ are the Hendricks-Manolescu involutive correction terms.  Recall that $\underline{d} \leq d \leq \overline{d}$ \cite{HM} and that $HF_{conn} = 0$ if and only if $\underline{d} = \overline{d}$ by \cite[Proposition 4.6]{HHL}.  If $Y$ is a Seifert homology sphere obtained by splicing Seifert fibers in positively oriented Seifert manifolds $Y_1$ and $Y_2$, then by Corollary~\ref{thm:splice} and the splice-additivity of $\overline{\mu}$ \cite{Sav2}, we see 
\[
\overline{d}(Y_1) + \overline{d}(Y_2) = d(Y_1) + d(Y_2) \geq d(Y) \geq -2\overline{\mu}(Y) = - 2\overline{\mu}(Y_1) -2\overline{\mu}(Y_2)  = \underline{d}(Y_1) + \underline{d}(Y_2).  
\]  
Since $HF_{conn}(Y_i) = 0$, we have $\overline{d}(Y_i) = \underline{d}(Y_i)$, and thus $d(Y) = -2\overline{\mu}(Y)$, i.e. $\underline{d}(Y) = \overline{d}(Y)$.  This implies that $HF_{conn}(Y) = 0$.   
\end{proof}

\section{Examples}\label{s:exam}
We illustrate Lemma~\ref{lem:destimate} with a basic family of graph manifold examples.  For each  $i \in \{ 1,2 \}$, pick a pair $p_i,q_i$ of relatively prime positive integers.  Additionally, let $m_1,m_2$ be any two positive integers.  For each $i$ let $Y_i$ be the positively oriented Seifert fibered integer homology sphere $\Sigma(p_i,q_i,p_iq_im_i+1)$ and let $K_i \subset Y_i$ be the singular fiber of order $p_iq_im_i +1$.

It is well-known that $Y_i(K_i;+1) \cong \Sigma(p_i,q_i,p_iq_i(m_i-1)+1)$ and that performing $-1$-surgery on the singular fiber of order $p_iq_i(m_i-1)+1$ in this latter manifold results in $\Sigma(p_i,q_i,p_iq_im_i+1) = Y_i$.   Now we can see that Proposition~\ref{pro:LT} above implies that $d(Y_i(K_i;\epsilon)) = d(Y_i)$ for each $i \in \{1,2\}$ and for each $\epsilon \in \{ -1, +1\}$.  Letting $Y$ be the splice of $Y_1$ and $Y_2$ along $K_1$ and $K_2$, we see that the upper and lower bounds in Lemma~\ref{lem:destimate} are both equal to $d(Y_1)+d(Y_2)$ and so $d(Y) = d(Y_1) + d(Y_2)$.  (This can also be deduced from Proposition~\ref{prop:knot-exteriors2} applied to the torus knots $T_{p_i,q_i}$ in $S^3$.)  

To extend the above example, consider a Seifert fibered homology sphere $\Sigma(a_1,a_2,\dots,a_n)$. Call the operation $\Sigma(a_1,a_2,\dots,a_n)\to \Sigma(a_1,a_2,\dots,a_n+a_1\dots a_{n-1})$ stabilization and its inverse destabilization of the singular fiber of order $r$.  We say that the singular fiber of order $a_n$ is \emph{ stabilized} if $a_n >a_1\dots a_{n-1}$.  In this case $- 1$- and $+1$-surgery on  the singular fiber of order $a_n$ correspond to stabilization and destabilization of the same singular fiber, so by Proposition~\ref{pro:LT} they do not change the $d$-invariant. Hence we can conclude that the $d$-invariant is splice additive if we splice two Seifert fibered homology spheres along stabilized singular fibers. Note that the same argument also works even if we reverse orientations on either (or both) of the Seifert homology spheres. As a result we can compute the $d$-invariant of those graph homology spheres whose splice diagram contains two nodes and spliced singular fibers are stabilized.
\begin{figure}[h]
	\includegraphics[width=0.30\textwidth]{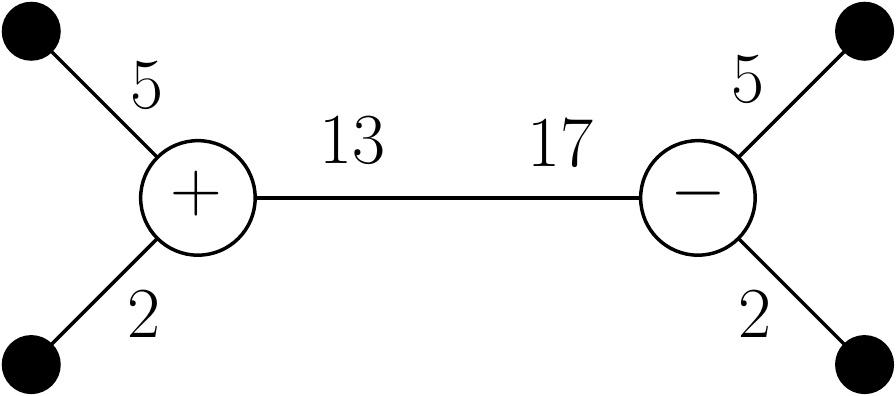}
	\caption{Splice diagram of $Y$}
	\label{fig:splice_diag}
\end{figure}

  For example, let $Y_1=\Sigma(2,5,13)$, with $K_1$ the singular fiber of order $13$, and let $Y_2=-\Sigma(2,5,17)$, with $K_2$ the singular fiber of order $17$.  Here, $\Sigma(2,5,13)$ has positive Seifert orientation (bounding a negative-definite plumbing), while $-\Sigma(2,5,17)$ is negatively oriented (bounding a positive-definite plumbing).  Let $Y$ be the splice of $Y_1$ and $Y_2$ along $K_1$ and $K_2$.  (See  Figure~\ref{fig:splice_diag} for the splice diagram of $Y$.) Note that the associated plumbing is neither positive-definite nor negative-definite.  Nonetheless, we compute 
\begin{align*}
d(Y)&=d(Y_1)+d(Y_2)\\
&=d(\Sigma(2,5,13))- d(\Sigma(2,5,17))\\
&=d(\Sigma(2,3,5))- d(\Sigma(2,5,7))\\
&=2-0\\
&=2.  
\end{align*}

Note that stabilization is essential for the splice additivity. For example $\Sigma(3,11,13,20)$ is the splice of $\Sigma(33,13,20)$ and $\Sigma(3,11,260)$ along singular fibers of orders $33$ and $260$ respectively. We have $d(\Sigma(3,11,13,20))=d(\Sigma(33,13,20))=d(\Sigma(3,11,260))=2$, so 
$$d(\Sigma(3,11,13,20)) \not = d(\Sigma(33,13,20))+d(\Sigma(3,11,260)).$$

\bibliographystyle{plain}
\bibliography{references}

\end{document}